\documentclass{amsart}
\usepackage{bbding}
\usepackage{amssymb}
\usepackage{amsmath}

\newtheorem{theorem}{Theorem}[section]

\newtheorem{proposition}[theorem]{Proposition}
\newtheorem{corollary}[theorem]{Corollary}
\theoremstyle{definition}
\newtheorem{definition}[theorem]{Definition}
\newtheorem{example}[theorem]{Example}

\newtheorem{remark}[theorem]{Remark}
\numberwithin{equation}{section}

\newcommand{\blankbox}[2]

\begin{document}
\title{A class of Lie conformal superalgebras in higher dimensions}
\thanks{Project supported by the Scientific Research Foundation of Zhejiang Agriculture and Forestry University (No.2013FR081)}
\author{Yanyong Hong}
\address{College of Science, Zhejiang Agriculture and Forestry University,
Hangzhou, 311300, P.R.China}
\email{hongyanyong2008@yahoo.com}


\begin{abstract}
Fix a positive integer number $r$. A class of $r$-dim Lie conformal superalgebras named $r$-dim $i$-linear Lie conformal superalgebras are studied for $1\leq i \leq r$. We present an equivalent characterization of this class of Lie conformal superalgebras in $r$ dimension. Moreover, by this equivalent characterization, several constructions and examples are given.
\end{abstract}

\maketitle

\section{Introduction}
Throughout this paper, set $r$ a positive integer number and denote by $\mathbb{C}$ the field of complex
numbers; $\mathbf{N}$ the set of natural numbers, i.e.
$\mathbf{N}=\{0, 1, 2,\cdots\}$; $\mathbb{Z}$ the set of integer
numbers. And,  tensors over $\mathbb{C}$ are denoted by  $\otimes$.

Lie conformal superalgebra, introduced by Kac in \cite{K1}, gives an axiomatic decription of the singular part of the operator product expansion of chiral fields in conformal field theory. It is an useful tool to study vertex superalgebras (see \cite{K1}) and has many applications in the theory of infinite-dimensional Lie superalgebras and integrable systems. In special, the category of Lie conformal superalgebras is equivalent to the category of some infinite-dimensional Lie superalgebras named formal distribution Lie superalgebras in some sense (see \cite{K1}). The main examples of Lie conformal superalgebras are Virasoro conformal algebra, current conformal superalgebras, Neveu-Schwarz conformal superalgebra and so on. It is known that Lie conformal superalgebra is a $\mathbb{C}[T]$-module. Therefore, there is a natural generalization, i.e. we change $\mathbb{C}[T]$ by $\mathbb{C}[T_1,\cdots,T_n]$. This is just the $n$-dim Lie conformal superalgebra (see \cite{BKV}). More general, $\mathbb{C}[T]$ can be generalized to a cocommutative Hopf algebra. This is the case of Lie pseudoalgebras (see \cite{BDK}).
 The
structure theory \cite{DK1}, representation theory \cite{CK1, CK2}
and cohomology theory \cite{BKV} of finite Lie conformal algebras
has been developed.

A class of Lie conformal superalgebras named ``quadratic" Lie conformal superalgebras are also studied in \cite{GD, Xu1}. It was essentially stated in \cite{GD} that a ``quadratic" Lie conformal superalgebra is equivalent to a bialgebra structure, one is a Lie superalgebra structure , the other is a Novikov superalgebra structure and they satisfy a compatible condition. The above bialgebra structure is called super Gel'fand-Dorfman bialgebra by Xu in \cite{Xu1}. In fact, a quadratic Lie conformal algebra corresponds to a Hamiltonian pair in \cite{GD}, which plays fundamental roles in completely integrable systems. Then, Xu in \cite{Xu1} presented several constructions of super Gel'fand-Dorfman bialgebras and classified the Gel'fand-Dorfman bialgebra structures on simple Novikov algebras. Therefore, there are some natural questions that whether there exists a similar characterization of Lie conformal superalgebras in higher dimensions and if there exists, what is the corresponding algebraic structure. In this paper, we study this question. We prove that an $r$-dim $i$-linear Lie conformal superalgebra (see Definition \ref{def1}) is equivalent to one $(r-1)$-dim conformal superalgebra with a Lie conformal superalgebra structure  and a Novikov conformal superalgebra structure adjoint with some compatible conditions. We call such structure $(r-1)$-dim super Gel'fand-Dorfman conformal bialgebra. In special, we also show that an $r$-dim linear Lie conformal superalgebra (see Definition \ref{def1}) is equivalent to one vector space with a Lie superalgebra structure and $r$ Novikov superalgebra structures adjoint with some compatible conditions. Here, any one of $r$ Novikov superalgebra structures and the Lie superalgebra structure construct a super Gel'fand-Dorfman bialgebra structure. We call such an algebraic structure a generalized super Gel'fand-Dorfman algebra. Moreover, we present several constructions and examples of super Gel'fand-Dorfman conformal bialgebras in $r$ dimension and generalized super Gel'fand-Dorfman algebras. Since one $r$-dim Lie conformal superalgebra corresponds to one infinite-dimensional Lie superalgebra, we can give many interesting examples of infinite-dimensional Lie superalgebras using super Gel'fand-Dorfman conformal bialgebras in $r$ dimension and generalized super Gel'fand-Dorfman algebras.

This paper is organized as follows. In Section 2, we introduce some basic definitions and some facts about $r$-dim Lie conformal superalgebras and Novikov conformal superalgebras. The definitions of $r$-dim $i$-linear Lie conformal superalgebras  and linear Lie conformal superalgebras are presented for $1\leq i\leq r$.
In Section 3, we mainly present two equivalent characterizations of $r$-dim $i$-linear Lie conformal superalgebras  and linear Lie conformal superalgebras. Several constructions and examples are given.

\section{Preliminaries}
In this section, we will introduce some basic definitions and some facts
about $r$-dim Lie conformal superalgebras and $r$-dim Novikov conformal superalgebras. These facts can be found in \cite{BKV, K1}.

First, we introduce the definitions of Lie superalgebra (see \cite{K2}) and Lie-Poisson superalgebra.
\begin{definition}
A \emph{Lie superalgebra} $\mathfrak{g}$ is a $\mathbb{Z}/2\mathbb{Z}$-graded algebra $\mathfrak{g}=\mathfrak{g}_{\overline{0}}\oplus \mathfrak{g}_{\overline{1}}$ with an operation $[\cdot, \cdot]$ satisfying the following axioms:
\begin{eqnarray*}
[a,b]=-(-1)^{\alpha\beta}[b,a],~~[a,[b,c]]=[[a,b],c]+(-1)^{\alpha\beta}[b,[a,c]],
\end{eqnarray*}
for $a\in \mathfrak{g}_{\alpha}$, $b\in \mathfrak{g}_{\beta}$ and $c\in \mathfrak{g}$.

A \emph{Lie-Poisson superalgebra} $(A, [\cdot,\cdot], \cdot)$ is a $\mathbb{Z}/2\mathbb{Z}$-graded vector space $A=A_{\overline{0}}\oplus A_{\overline{1}}$ with two operations $[\cdot, \cdot]$ and $\cdot$ such that
$(A, [\cdot,\cdot])$ forms a Lie superalgebra and $(A,\cdot)$ forms a commutative associative superalgebra with the compatible condition:
\begin{eqnarray*}
[a, b\cdot c]=[a,b]\cdot c+(-1)^{\alpha\beta}b\cdot [a,c],  ~~~~a\in A_\alpha, ~~b\in A_\beta, c\in A.
\end{eqnarray*}
If $A_{\overline{1}}=\{0\}$, it is called \emph{Lie-Poisson algebra}.
\end{definition}

Set $\mathbf{T}=(T_1,\cdots,T_r)$ and $\mathbf{\lambda}=(\lambda_1,\cdots,\lambda_r)$. Then, $\mathbf{T}$ and $\mathbf{\lambda}$ are $r$ dimensional vector variables. Assume that $\mathbf{m}=(m_1,\cdots, m_r)\in \mathbb{Z}^r$ and $e_i=(\underbrace{0, \cdots, 1}_i, \cdots, 0)\in\mathbb{Z}^r$ for $1\leq i\leq r$. Set $\mathbb{C}[\mathbf{T}]=\mathbb{C}[T_1,\cdots, T_r]$.

Next, we introduce the definition of $r$-dim Lie conformal superalgebra.

\begin{definition}
An \emph{$r$-dim Lie conformal superalgebra} is a left $\mathbb{Z}/2\mathbb{Z}$-graded $\mathbb{C}[\mathbf{T}]$-module $R=R_{\overline{0}}\oplus R_{\overline{1}}$ endowed with a $\mathbf{\lambda}$-bracket $[a_{\mathbb{\lambda}}b]$ which defines a $\mathbb{C}$-bilinear map $R\otimes R\rightarrow R[\mathbf{\lambda}]$, where $R[\mathbf{\lambda}]=\mathbb{C}[\mathbf{\lambda}]\otimes R=\mathbb{C}[\lambda_1,\cdots, \lambda_r]\otimes R$, subject to the following axioms:
\begin{eqnarray}
&&\label{a6}[T_ia_{\mathbf{\lambda}}b]=-\lambda_i[a_{\mathbf{\lambda}}b],\qquad [a_{\mathbf{\lambda}}T_ib]=(T_i+\lambda_i)[a_{\mathbf{\lambda}}b],\\
&& \label{a7}[a_{\mathbf{\lambda}}b]=-(-1)^{\alpha\beta}[b_{-(\mathbf{\lambda}+\mathbf{T})}a],\\
&&\label{a8} [a_{\mathbf{\lambda}}[b_{\mathbf{\mu}}c]]=[[a_{\mathbf{\lambda}}b]_{\mathbf{\lambda}+\mathbf{\mu}}c]+(-1)^{\alpha\beta}[b_{\mathbf{\mu}}[a_{\mathbf{\lambda}}c]],
\end{eqnarray}
for $a\in R_{\alpha}$, $b\in R_{\beta}$ and $c\in R$. When $R_{\overline{1}}=\{0\}$, $R$ is called \emph{$r$-dim Lie conformal algebra}. If $r=1$,
$R$ is called \emph{Lie conformal superalgebra}.
\end{definition}

If $\mathbf{\lambda}=(\lambda_1, \cdots, \lambda_r)$ and $\mathbf{m}=(m_1,\cdots,m_r)$, we set $\mathbf{\lambda}^{(\mathbf{m})}=\lambda_1^{(m_1)}\cdots\lambda_r^{(m_r)}=\frac{\lambda_1^{m_1}}{(m_1)!}\cdots\frac{\lambda_r^{m_r}}{(m_r)!}$.
We can rewrite
$$[a_{\mathbf{\lambda}}b]=\sum_{\mathbf{m}\in \mathbb{Z}_{+}^r}\mathbf{\lambda}^{(\mathbf{m})}a_{(\mathbf{m})}b.$$

Given an $r$-dim Lie conformal superalgebra $R$, we can associate a Lie superalgebra as follows.
Let Lie$R$ be the quotient
of the $\mathbb{Z}/2\mathbb{Z}$-graded vector space with basis $a_{\mathbf{n}}$ $(a\in R, \mathbf{n}\in\mathbb{Z}^r)$ by
the subspace spanned over $\mathbb{C}$ by
elements:
$$(\alpha a)_{\mathbf{n}}-\alpha a_{\mathbf{n}},~~(a+b)_{\mathbf{n}}-a_{\mathbf{n}}-b_{\mathbf{n}},~~(T_i
a)_{\mathbf{n}}+n_ia_{{\mathbf{n}}-e_i},~~~\text{where}~~a,~~b\in R,~~\alpha\in \mathbb{C},~~{\mathbf{n}}\in
\mathbb{Z}^r.$$
Here, if $a\in R_{\alpha}$, $a_{\mathbf{n}}\in $ $({\text{Lie} R})_{\alpha}$. The operation on Lie$R$ is defined as follows:
\begin{eqnarray}\label{106}
[a_{\mathbf{m}}, b_\mathbf{n}]=\sum_{\mathbf{j}\in \mathbb{Z}_{+}^r}\left(\begin{array}{ccc}
\mathbf{m}\\\mathbf{j}\end{array}\right)(a_{(\mathbf{j})}b)_{\mathbf{m}+\mathbf{n}-\mathbf{j}},\end{eqnarray}
where $\left(\begin{array}{ccc}
\mathbf{m}\\\mathbf{j}\end{array}\right)=\left(\begin{array}{ccc}
m_1\\j_1\end{array}\right)\cdots\left(\begin{array}{ccc}
m_r\\j_r\end{array}\right)$, if $\mathbf{m}=(m_1,\cdots,m_r)$ and $\mathbf{j}=(j_1,\cdots,j_r)$.

\begin{example}\label{a2}
Let $\mathfrak{g}$ be a Lie superalgebra. Then $\text{Cur} \mathfrak{g}=\mathbb{C}[\mathbf{T}]\otimes \mathfrak{g}$ can be endowed with the following $\mathbf{\lambda}$-bracket $(a, b\in \mathfrak{g})$:
$$[a_{\mathbf{\lambda}}b]=[a,b].$$
$\widetilde{\mathfrak{g}}=$Lie$(\text{Cur} \mathfrak{g})$ is just the current Lie superalgebra in $r$ variables, i.e. $\widetilde{\mathfrak{g}}=\mathfrak{g}\otimes_{\mathbb{C}}\mathbb{C}[x_1,x_1^{-1},\cdots,x_r,x_r^{-1}]$ with the following Lie bracket:
$$[a\otimes x_1^{m_1}\cdots x_r^{m_r}, b\otimes x_1^{n_1}\cdots x_r^{n_r}]=[a,b]\otimes x_1^{m_1+n_1}\cdots x_r^{m_r+n_r},$$
where $a$, $b\in \mathfrak{g}$, and $m_1$, $\cdots$, $m_r$, $n_1$, $\cdots$, $n_r\in \mathbb{Z}$.
\end{example}

\begin{example}\label{a3}
Let $R=\bigoplus_{i=1}^r\mathbb{C}[\mathbf{T}]L^i$. It can be endowed with the following $\mathbf{\lambda}$-bracket:
$$[L^i_{\mathbf{\lambda}}L^j]=T_iL^j+\lambda_iL^j+\lambda_jL^i.$$
Lie$R$ is just the Lie algebra $W_r=\text{Der} \mathbb{C}[x_1,x_1^{-1},\cdots,x_r,x_r^{-1}]$. For $r=1$, $R$ is the Virasoro Lie conformal algebra.
\end{example}

\begin{example}\label{aa4}
Let $r=2s$ and $R=\mathbb{C}[\mathbf{T}]L$. The following $\mathbf{\lambda}$-bracket endows $R$ a Lie conformal algebra structure:
$$[L_{\mathbf{\lambda}}L]=\sum_{i=1}^s(\lambda_{s+i}T_iL-\lambda_iT_{s+i}L).$$
Lie$R$ is just the subalgebra $H_r$ of Hamiltonian derivations of $W_r$.
\end{example}

In the following of this paper, we mainly study a special class of $r$-dim Lie conformal superalgebras.

Let $\overline{\lambda}_i=(\lambda_1,\cdots,\lambda_{i-1},\lambda_{i+1},\cdots,\lambda_{r})$ and $\overline{\mathbf{T}}_i=(T_1,\cdots,T_{i-1}, T_{i+1},\cdots,T_{r})$ for $i\in\{1,\cdots, r\}$.

\begin{definition}\label{def1}
Suppose that $R$ is a free $\mathbb{C}[\mathbf{T}]$-module over a $\mathbb{Z}/2\mathbb{Z}$-graded vector space $V$. $R$ is called an \emph{$r$-dim $i$-linear Lie conformal superalgebra} if for any $a$, $b\in V$, its $\mathbf{\lambda}$-bracket is of the following form:
\begin{eqnarray}\label{b1}
[a_{\mathbf{\lambda}}b]=T_if(\overline{\lambda}_i,\overline{\mathbf{T}}_i)u+\lambda_ig(\overline{\lambda}_i,\overline{\mathbf{T}}_i)v+h(\overline{\lambda}_i,\overline{\mathbf{T}}_i)w,\end{eqnarray}
where $u$, $v$, $w\in V$, $f(\overline{\lambda}_i,\overline{\mathbf{T}}_i)$, $g(\overline{\lambda}_i,\overline{\mathbf{T}}_i)$, $h(\overline{\lambda}_i,\overline{\mathbf{T}}_i)\in \mathbb{C}[\lambda_1,\cdots, \lambda_{i-1}, \lambda_{i+1},\cdots, \lambda_r,$ $ T_1,\cdots,T_{i-1},T_{i+1},\cdots,T_{r}]$, $i\in\{1,\cdots,r\}$. In special,
$R$ is called an \emph{$r$-dim linear Lie conformal superalgebra} if for any $a$, $b\in V$, its $\mathbf{\lambda}$-bracket is of the following form:
\begin{eqnarray}\label{a4}
[a_{\mathbf{\lambda}}b]=\sum_{i=1}^r\lambda_iw_i+\sum_{j=1}^rT_ju_j+v,~~~~~u_i,~w_i,~~v\in V.
\end{eqnarray}
When $r=1$, for simplicity, it is also called \emph{linear Lie conformal superalgebra}.
\end{definition}

\begin{remark}
Since the right side of (\ref{b1}) can be seen a linear function with two variables $\lambda_i$, $T_i$, we take the name ``$r$-dim $i$-linear Lie conformal superalgebra". Similarly, since the right side of (\ref{a4}) can be seen a linear function with variables $\lambda_1$, $\cdots$, $\lambda_r$, $T_1$, $\cdots$, $T_r$, we call it ``$r$-dim linear Lie conformal superalgebra". When $r=1$, it is called ``quadratic Lie conformal superalgebra" in \cite{Xu1}, since the definition of Lie conformal superalgebra is given by the language of formal distribution (see \cite{K1}).
\end{remark}

\begin{remark}
Obviously, Examples \ref{a2} and \ref{a3} are $r$-dim linear Lie conformal superalgebras.
Example \ref{aa4} is an $r$-dim $r$-linear Lie conformal superalgebra.
\end{remark}

Next, we introduce the definitions of Novikov superalgebra and Novikov-Poisson superalgebra.
\begin{definition}
A (left) \emph{Novikov superalgebra} is a $\mathbb{Z}/2\mathbb{Z}$-graded vector space $A=A_{\overline{0}}\oplus A_{\overline{1}}$ with an operation
``$\circ$" satisfying the following axioms: for $a\in A_{\alpha}$, $b\in A_{\beta}$, $c\in A_{\gamma}$,
\begin{eqnarray*}
&(a\circ b)\circ c=(-1)^{\beta\gamma}(a\circ c)\circ b,\\
&(a\circ b)\circ c-a\circ (b\circ c)=(-1)^{\alpha\beta}[(b\circ a)\circ c-b\circ(a\circ c)].
\end{eqnarray*}
When $A_{\overline{1}}=0$, we call $A$ a (left) \emph{Novikov algebra}.

A \emph{Novikov-Poisson superalgebra} is a $\mathbb{Z}/2\mathbb{Z}$-graded vector space $A=A_{\overline{0}}\oplus A_{\overline{1}}$ with two operations
``$\circ$" and ``$\cdot$", where $(A,\circ)$ is a Novikov superalgebra, $(A,\cdot)$ is a commutative associative superalgebra, and they satisfy the following axioms:
\begin{eqnarray*}
(a\circ b)\cdot c-a\circ(b\cdot c)=(-1)^{\alpha\beta}((b\circ a)\cdot c-b\circ(a\cdot c)),\\
(a\cdot b)\circ c=a\cdot(b\circ c),
\end{eqnarray*}
for $a\in A_{\alpha}$, $b\in A_{\beta}$. When $A_{\overline{1}}=\{0\}$, it is called \emph{Novikov-Poisson algebra}.
\end{definition}

\begin{remark}
Novikov algebra was essentially stated in \cite{GD} that it corresponds to a
certain Hamiltonian operator. Such an algebraic structure appeared
in \cite{BN} from the point of view of Poisson structures of
hydrodynamic type. The name ``Novikov algebra" was given by Osborn
\cite{Os}. Moreover, Novikov-Poisson algebra is introduced by Xu in \cite{X2} to study Novikov algebra, and the author also studied the constructions of Novikov-Poisson algebra in \cite{X3}.
\end{remark}

\begin{definition}
An \emph{$r$-dim right Novikov conformal superalgebra} is a left $\mathbb{Z}/2\mathbb{Z}$-graded $\mathbb{C}[\mathbf{T}]$-module $R=R_{\overline{0}}\oplus R_{\overline{1}}$ endowed with a $\mathbb{\lambda}$-product $a_{\mathbb{\lambda}}b$ which defines a $\mathbb{C}$-bilinear map $R\otimes R\rightarrow R[\mathbf{\lambda}]$, where $R[\mathbf{\lambda}]=\mathbb{C}[\mathbf{\lambda}]\otimes R=\mathbb{C}[\lambda_1,\cdots, \lambda_r]\otimes R$, subject to the following axioms:
\begin{eqnarray}
&&(T_ia_{\mathbf{\lambda}}b)=-\lambda_i(a_{\mathbf{\lambda}}b),\qquad (a_{\mathbf{\lambda}}T_ib)=(T_i+\lambda_i)(a_{\mathbf{\lambda}}b),\\
&&a_{\mathbf{\lambda}}(b_{\mathbf{\mu}}c)-(a_{\mathbf{\lambda}}b)_{\mathbf{\lambda}+\mathbf{\mu}}c=(-1)^{\beta\gamma}(a_{\mathbf{\lambda}}(c_{-\mathbf{\mu}-\mathbf{T}}b)-(a_{\mathbf{\lambda}}c)_{-\mathbf{\mu}-\mathbf{T}}b),\\ &&a_{\mathbf{\lambda}}(b_{\mathbf{\mu}}c)=(-1)^{\alpha\beta}b_{\mathbf{\mu}}(a_{\mathbf{\lambda}}c).
\end{eqnarray}
for $a\in R_{\alpha}$, $b\in R_{\beta}$ and $c\in R_{\gamma}$. When $R_{\overline{1}}=\{0\}$, $R$ is called \emph{$r$-dim right Novikov conformal algebra}. If $r=1$, it is called \emph{right Novikov conformal superalgebra}.

An \emph{$r$-dim (left) Novikov conformal superalgebra} is a left $\mathbb{Z}/2\mathbb{Z}$-graded $\mathbb{C}[\mathbf{T}]$-module $R=R_{\overline{0}}\oplus R_{\overline{1}}$ endowed with a $\mathbb{\lambda}$-product $a_{\mathbb{\lambda}}b$ which defines a $\mathbb{C}$-bilinear map $R\otimes R\rightarrow R[\mathbf{\lambda}]$, where $R[\mathbf{\lambda}]=\mathbb{C}[\mathbf{\lambda}]\otimes R=\mathbb{C}[\lambda_1,\cdots, \lambda_r]\otimes R$, subject to the following axioms:
\begin{eqnarray}
&&(T_ia_{\mathbf{\lambda}}b)=-\lambda_i(a_{\mathbf{\lambda}}b),\qquad (a_{\mathbf{\lambda}}T_ib)=(T_i+\lambda_i)(a_{\mathbf{\lambda}}b),\\
&&(a_\mathbf{\lambda}b)_{\mathbf{\lambda}+\mathbf{\mu}}c-a_{\mathbf{\lambda}}(b_{\mathbf{\mu}}c)=(-1)^{\alpha\beta}((b_{\mathbf{\mu}}a)_{\mathbf{\lambda}+\mathbf{\mu}}c-b_{\mathbf{\mu}}(a_{\mathbf{\lambda}}c)),\\
&&(a_{\mathbf{\lambda}}b)_{\mathbf{\lambda}+\mathbf{\mu}}c=(-1)^{\beta\gamma}(a_{\mathbf{\lambda}}c)_{-\mathbf{\mu}-\mathbf{T}}b,
\end{eqnarray}
for $a\in R_{\alpha}$, $b\in R_{\beta}$ and $c\in R_{\gamma}$. When $R_{\overline{1}}=\{0\}$, $R$ is called \emph{$r$-dim (left) Novikov conformal algebra}. If $r=1$, it is called \emph {
(left) Novikov conformal superalgebra}.

\end{definition}
\begin{remark}\label{r1}
The definition of {(left) Novikov conformal algebra} is introduced in \cite{HL1}. Obviously, if $(R,\cdot_{\lambda}\cdot)$ is a left (resp. right) Novikov conformal superalgebra, then $(R, \circ_{\lambda})$ is a right (resp. left) Novikov conformal superalgebra with
$a\circ_{\lambda}b=(-1)^{\alpha\beta}b_{-\lambda-T}a$ for any $a\in R_{\alpha}$, $b\in R_{\beta}$.
\end{remark}

\begin{example}\label{ex1}
If $(A,\circ)$ is a Novkiov superalgebra, then $R=\mathbb{C}[T]A$ is a Novikov conformal superalgebra with the following $\lambda$-product:
\begin{eqnarray*}
a_{\lambda}b=a\cdot b,~~~~a,~~b\in A.
\end{eqnarray*}

Let $R=R_{\overline{0}}\oplus R_{\overline{1}}=\mathbb{C}[T]a\oplus\mathbb{C}[T]b$ be a free $\mathbb{C}[T]$-module generated by $a$, $b$. Then, $R$ is a Novikov conformal superalgebra with the following $\lambda$-product:
\begin{eqnarray*}
a_{\lambda}a=(\lambda+T+C_1)a,~~~a_{\lambda}b=(\lambda+T+C_1)b,~~~b_{\lambda}a=(\lambda+T+C_2)b,~~b_{\lambda}b=0,
\end{eqnarray*}
where $C_1$, $C_2\in \mathbb{C}$.
\end{example}

Next, we present a theorem about the classification of torsion-free Novikov conformal algebras of rank 1 in \cite{HL1}.
\begin{proposition}\label{pp3}
Let $R=\mathbb{C}[T]x$ be a Novikov conformal algebra
that is free of rank one as a $\mathbb{C}[T]$-module. Then
$R$ is isomorphic to the Novikov conformal algebra with $\lambda$-product which is one of three cases as follows:\\
(i)\quad$x_{\lambda}x=0;$\\
(ii)\quad $x_{\lambda}x=x,~~~~i.e.~~\text{R~~is~~associative};$\\
(iii)\quad$x_{\lambda}x=(\lambda+T+a)x,~~ \text{for~~ any}~~ a\in
\mathbb{C}, ~~\text{denoted ~~by }~~V_a.$
\end{proposition}

Next, we also introduce a construction of  Novkiov conformal superalgebras from Novkov-Poisson superalgebras (see \cite{HL1}).
\begin{proposition}\label{pp2}
If $(V, \circ, \cdot)$ is a Novikov-Poisson superalgebra, then $R=\mathbb{C}[T]V$ is a Novikov conformal superalgebra with the following $\lambda$-product:
\begin{eqnarray}
a_{\lambda}b=(\lambda+T)(a\cdot b)+a\circ b, ~~~~a\in V_{\alpha},~~b\in V_{\beta}.
\end{eqnarray}
\end{proposition}

In addition, it is easy to check the following proposition.
\begin{proposition}\label{ppp1}
If $(R,\cdot_{\lambda}\cdot)$ is an $r$-dim left (resp. right) Novikov conformal superalgebra, then
$(R, [\cdot_{\lambda}\cdot])$ is an $r$-dim Lie conformal superalgebra  with the following $\lambda$-bracket:
\begin{eqnarray*}
[a_{\lambda} b]=a_{\lambda}b-(-1)^{\alpha\beta}(b_{-\lambda-\mathbf{T}}a),~~~~a\in R_{\alpha},~~b\in R_{\beta}.
\end{eqnarray*}
\end{proposition}

\begin{definition}(see \cite{Xu1})
A \emph{super Gel'fand-Dorfman bialgebra} is a $\mathbb{Z}/2\mathbb{Z}$-graded vector space $\mathfrak{A}=\mathfrak{A}_{\overline{0}}\oplus
\mathfrak{A}_{\overline{1}}$ with two algebraic operations $[\cdot,\cdot]$ and $\circ$ such that $(\mathfrak{A},[\cdot,\cdot])$ forms a Lie superalgebra, $(\mathfrak{A},\circ)$ forms a Novikov superalgebra and the following compatibility condition holds:
\begin{eqnarray}\label{xx1}
[a\circ b, c]+[a,b]\circ c-a\circ [b,c]-(-1)^{\beta\gamma}[a\circ c, b]\end{eqnarray}
$$-(-1)^{\beta\gamma}[a,c]\circ b=0,$$
for $a\in \mathfrak{A}$, $b\in \mathfrak{A}_{\beta}$, and $c\in \mathfrak{A}_{\gamma}$. When $\mathfrak{A}_{\overline{1}}=0$, it is called \emph{Gel'fand-Dorfman bialgebra} (see also \cite{Xu1}).
\end{definition}

A characterization of linear Lie conformal superalgebra is as follows.
\begin{theorem}\label{TTa1}(see \cite{Xu1},\cite{GD})
A linear Lie conformal superalgebra $R=\mathbb{C}[T]V$ is equivalent to that $V$ is a super Gel'fand-Dorfman bialgebra.
\end{theorem}
\section{Characterization of $i$-linear Lie conformal superalgebra in $r$ dimension}
In this section, we will give characterizations of $r$-dim $i$-linear Lie conformal superalgebra and $r$-dim linear Lie conformal superalgebra
for $i\in\{1,\cdots,r\}$.

\begin{theorem}\label{TT1}
An $r$-dim $i$-linear Lie conformal superalgebra $R=\mathbb{C}[\mathbf{T}]V$ is equivalent to that $\overline{R}_i=\mathbb{C}[\overline{\mathbf{T}}_i]V$ is equipped an $(r-1)$-dim Novikov conformal superalgebra structure with the $\overline{\mathbf{\lambda}}_i$-product $a_{\overline{\mathbf{\lambda}}_i}b$, an $(r-1)$-dimensional Lie conformal superalgebra structure and they satisfy the following condition:
\begin{eqnarray}\label{ccc1}\qquad
[(a_{-\overline{\mathbf{\mu}}_i-\overline{\mathbf{T}}_i}b)_{-\overline{\mathbf{\lambda}}_i-\overline{\mathbf{T}}_i}c]
+[a_{-\overline{\mathbf{\mu}}_i-\overline{\mathbf{T}}_i}b]_{-\overline{\mathbf{\lambda}}_i-\overline{\mathbf{T}}_i}c
-a_{-\overline{\mathbf{\lambda}}_i-\overline{\mathbf{\mu}}_i-\overline{\mathbf{T}}_i}[b_{-\overline{\mathbf{\lambda}}_i-\overline{\mathbf{T}}_i}c]
\end{eqnarray}
$$-(-1)^{\beta\gamma}[(a_{-\overline{\mathbf{\lambda}}_i-\overline{\mathbf{T}}_i}c)_{-\overline{\mathbf{\mu}}_i-\overline{\mathbf{T}}_i}b]
-(-1)^{\beta\gamma}[a_{-\overline{\mathbf{\lambda}}_i-\overline{\mathbf{T}}_i}c]_{-\overline{\mathbf{\mu}}_i-\overline{\mathbf{T}}_i}b=0,$$
for $a\in V$, $b\in V_{\beta}$ , $c\in V_{\gamma}$. We call such $\overline{R}_i=\mathbb{C}[\overline{\mathbf{T}}_i]V$ an \emph{$(r-1)$-dim super Gel'fand-Dorfman conformal bialgebra}.
\end{theorem}
\begin{proof}
Obviously, we only need to prove this theorem for one special case, i.e. $i=r$. Other cases can be obtained similarly.

For convenience, set $\overline{\mathbf{\lambda}}=\overline{\mathbf{\lambda}}_r$, $\overline{\mathbf{T}}=\overline{\mathbf{T}}_r$,
$\overline{\mathbf{\mu}}=\overline{\mathbf{\mu}}_r$ and $\overline{R}=\overline{R}_r=\mathbb{C}[\overline{\mathbf{T}}_r]V$.

By (\ref{b1}), we can assume that
\begin{eqnarray}\label{bb1}
[a_{\mathbf{\lambda}}b]=T_r(a\circ_{\overline{\lambda}}b)+\lambda_r(a\ast_{\overline{\lambda}}b)+[a_{\overline{\lambda}}b],\qquad a,~~b\in V.
\end{eqnarray}
Here, $\circ_{\overline{\lambda}}$, $\ast_{\overline{\lambda}}$ and $[\cdot_{\overline{\lambda}}\cdot]$ are three $\mathbb{C}$-bilinear maps from $\overline{R}\otimes \overline{R}\rightarrow \mathbb{C}[\overline{\lambda}]\otimes \overline{R}$ where
$\overline{R}=\mathbb{C}[\overline{\mathbf{T}}]V$.

Since $[T_ia_{\lambda} b]=-\lambda_i[a_\lambda b]$ and $[a_\lambda T_i b]=(\lambda_i+T_i)[a_\lambda b]$ for
$1\leq i\leq r-1$, by comparing the coefficients of $T_r$, $\lambda_r$ and $T_r^{0}\lambda_r^{0}$, we can obtain that
\begin{eqnarray}
&&(T_i a)\circ_{\overline{\lambda}} b=-\lambda_i (a\circ_{\overline{\lambda}} b), \qquad a\circ_{\overline{\lambda}}(T_i b)=(\lambda_i+T_i)(a\circ_{\overline{\lambda}}b),\\
&&(T_i a)\ast_{\overline{\lambda}} b=-\lambda_i (a\ast_{\overline{\lambda}} b),\qquad a\ast_{\overline{\lambda}}(T_i b)=(\lambda_i+T_i)(a\ast_{\overline{\lambda}}b),\\
&&[(T_ia)_{\overline{\lambda}} b]=-\lambda_i[a_{\overline{\lambda }}b],\qquad [a_{\overline{\lambda}}T_i b]=(\lambda_i+T_i)[a_{\overline{\lambda}} b],
\end{eqnarray}
 for $1\leq i\leq r-1$.

For any $a\in V_{\alpha}$, $b\in V_{\beta}$, using (\ref{b1}), (\ref{a7}) becomes:
\begin{eqnarray*}
&&T_r(a\circ_{\overline{\lambda}}b)+\lambda_r(a\ast_{\overline{\lambda}}b)+[a_{\overline{\lambda}}b]\\
&&=-(-1)^{\alpha\beta}(T_r(b\circ_{-\overline{\lambda}-\overline{\mathbf{T}}}a)-(\lambda_r+T_r)(b\ast_{-\overline{\lambda}-\overline{T}}a)+[b_{-\overline{\lambda}-\overline{\mathbf{T}}}a]).
\end{eqnarray*}
Since $R$ is a free $\mathbb{C}[\mathbf{T}]$-module, by comparing the coefficients of $T_r$, $\lambda_r$ and $T_r^{0}\lambda_r^{0}$, we get
\begin{eqnarray}
&&\label{bb2}a\ast_{\overline{\lambda}}b=(-1)^{\alpha\beta}b\ast_{-\overline{\lambda}-\overline{\mathbf{T}}}a,\\
&&\label{bb3}b\ast_{-\overline{\lambda}-\overline{\mathbf{T}}}a=b\circ_{-\overline{\lambda}-\overline{\mathbf{T}}}a+(-1)^{\alpha\beta}a\circ_{\overline{\lambda}}b,\\
&&\label{bb4}[a_{\overline{\lambda}}b]=-(-1)^{\alpha\beta}[b_{-\overline{\lambda}-\overline{\mathbf{T}}}a].
\end{eqnarray}
By (\ref{bb1}) and (\ref{bb3}), we obtain
\begin{eqnarray}\label{bbb3}
a\ast_{\overline{\lambda}}b=a\circ_{\overline{\lambda}}b+(-1)^{\alpha\beta}b\circ_{-\overline{\lambda}-\overline{\mathbf{T}}}a.
\end{eqnarray}
Obviously, (\ref{bb2}) can be obtained from (\ref{bbb3}).

Next, we consider (\ref{a8}). For any $a\in V_{\alpha}$, $b\in V_{\beta}$, $c\in V_{\gamma}$, we have
\begin{eqnarray*}
&&[a_{\mathbf{\lambda}}[b_{\mathbf{\mu}}c]]\\
&=&[a_{\mathbf{\lambda}}(T_r(b\circ_{\overline{\mu}}c)+\mu_r(b\ast_{\overline{\mu}}c)+[b_{\overline{\mu}}c])]\\
&=&(\lambda_r+T_r)[a_\lambda(b\circ_{\overline{\mu}}c)]+\mu_r[a_\lambda(b\ast_{\overline{\mu}}c)]+[a_\lambda[b_{\overline{\mu}}c]]\\
&=&(\lambda_r+T_r)(T_r(a\circ_{\overline{\lambda}}(b\circ_{\overline{\mu}}c))
+\lambda_r(a\ast_{\overline{\lambda}}(b\circ_{\overline{\mu}}c))+[a_{\overline{\lambda}}(b\circ_{\overline{\mu}}c)])\\
&&+\mu_r(T_r(a\circ_{\overline{\lambda}}(b\ast_{\overline{\mu}}c))
+\lambda_r(a\ast_{\overline{\lambda}}(b\ast_{\overline{\mu}}c))+[a_{\overline{\lambda}}(b\ast_{\overline{\mu}}c)])\\
&&+T_r(a\circ_{\overline{\lambda}}[b_{\overline{\mu}}c])
+\lambda_r(a\ast_{\overline{\lambda}}[b_{\overline{\mu}}c])+[a_{\overline{\lambda}}[b_{\overline{\mu}}c)]]\\
&=&T_r^2(a\circ_{\overline{\lambda}}(b\circ_{\overline{\mu}}c))+\lambda_rT_r(a\circ_{\overline{\lambda}}(b\circ_{\overline{\mu}}c)+
a\ast_{\overline{\lambda}}(b\circ_{\overline{\mu}}c))\\
&&+\lambda_r\mu_r(a\ast_{\overline{\lambda}}(b\ast_{\overline{\mu}}c))+\mu_rT_r(a\circ_{\overline{\lambda}}(b\ast_{\overline{\mu}}c))
+\lambda_r^2(a\ast_{\overline{\lambda}}(b\circ_{\overline{\mu}}c))\\
&&+\lambda_r([a_{\overline{\lambda}}(b\circ_{\overline{\mu}}c)]+
a\ast_{\overline{\lambda}}[b_{\overline{\mu}}c])+T_r([a_{\overline{\lambda}}(b\circ_{\overline{\mu}}c)]+a\circ_{\overline{\lambda}}[b_{\overline{\mu}}c])\\
&&+\mu_r[a_{\overline{\lambda}}(b\ast_{\overline{\mu}}c)]+[a_{\overline{\lambda}}[b_{\overline{\mu}}c]].
\end{eqnarray*}

Similarly, we can obtain
\begin{eqnarray*}
&&[b_{\mathbf{\mu}}[a_{\mathbf{\lambda}}c]]\\
&=&T_r^2(b\circ_{\overline{\mu}}(a\circ_{\overline{\lambda}}c))+\mu_rT_r(b\circ_{\overline{\mu}}(a\circ_{\overline{\lambda}}c)+
b\ast_{\overline{\mu}}(a\circ_{\overline{\lambda}}c))\\
&&+\lambda_r\mu_r(b\ast_{\overline{\mu}}(a\ast_{\overline{\lambda}}c))+\lambda_rT_r(b\circ_{\overline{\mu}}(a\ast_{\overline{\lambda}}c))
+\mu_r^2(b\ast_{\overline{\mu}}(a\circ_{\overline{\lambda}}c))\\
&&+\mu_r([b_{\overline{\mu}}(a\circ_{\overline{\lambda}}c)]+
b\ast_{\overline{\mu}}[a_{\overline{\lambda}}c])+T_r([b_{\overline{\mu}}(a\circ_{\overline{\mu}}c)]+b\circ_{\overline{\mu}}[a_{\overline{\lambda}}c])\\
&&+\lambda_r[b_{\overline{\mu}}(a\ast_{\overline{\lambda}}c)]+[b_{\overline{\mu}}[a_{\overline{\lambda}}c]].
\end{eqnarray*}

On the other hand, we have
\begin{eqnarray*}
&&[[a_\lambda b]_{\lambda+\mu}c]\\
&=&[(T_r(a\circ_{\overline{\lambda}}b)+\lambda_r(a\ast_{\overline{\lambda}}b)+[a_{\overline{\lambda}}b])_{\lambda+\mu}c]\\
&=&(-\lambda_r-\mu_r)[(a\circ_{\overline{\lambda}}b)_{\lambda+\mu}c]+\lambda_r[(a\ast_{\overline{\lambda}}b)_{\lambda+\mu}c]
+[[a_{\overline{\lambda}}b]_{\lambda+\mu}c]\\
&=&(-\lambda_r-\mu_r)(T_r((a\circ_{\overline{\lambda}}b)\circ_{\overline{\lambda}+\overline{\mu}}c)
+(\lambda_r+\mu_r)((a\circ_{\overline{\lambda}}b)\ast_{\overline{\lambda}+\overline{\mu}}c)
+[(a\circ_{\overline{\lambda}}b)_{\overline{\lambda}+\overline{\mu}}c])\\
&&+\lambda_r(T_r((a\ast_{\overline{\lambda}}b)\circ_{\overline{\lambda}+\overline{\mu}}c)
+(\lambda_r+\mu_r)((a\ast_{\overline{\lambda}}b)\ast_{\overline{\lambda}+\overline{\mu}}c)
+[(a\ast_{\overline{\lambda}}b)_{\overline{\lambda}+\overline{\mu}}c])\\
&&+T_r([a_{\overline{\lambda}}b]\circ_{\overline{\lambda}+\overline{\mu}}c)
+(\lambda_r+\mu_r)([a_{\overline{\lambda}}b]\ast_{\overline{\lambda}+\overline{\mu}}c)
+[[a_{\overline{\lambda}}b]_{\overline{\lambda}+\overline{\mu}}c]\\
&=&\lambda_r^2((a\ast_{\overline{\lambda}}b)\ast_{\overline{\lambda}+\overline{\mu}}c-(a\circ_{\overline{\lambda}}b)\ast_{\overline{\lambda}+\overline{\mu}}c)
+\lambda_rT_r((a\ast_{\overline{\lambda}}b)\circ_{\overline{\lambda}+\overline{\mu}}c-(a\circ_{\overline{\lambda}}b)\circ_{\overline{\lambda}+\overline{\mu}}c)\\
&&+\lambda_r\mu_r((a\ast_{\overline{\lambda}}b)\ast_{\overline{\lambda}+\overline{\mu}}c-2(a\circ_{\overline{\lambda}}b)\ast_{\overline{\lambda}+\overline{\mu}}c)
-\mu_rT_r(a\circ_{\overline{\lambda}}b)\circ_{\overline{\lambda}+\overline{\mu}}c\\
&&-\mu_r^2(a\circ_{\overline{\lambda}}b)\ast_{\overline{\lambda}+\overline{\mu}}c+\lambda_r([(a\ast_{\overline{\lambda}}b)_{\overline{\lambda}+\overline{\mu}}c]-[(a\circ_{\overline{\lambda}}b)_{\overline{\lambda}+\overline{\mu}}c]+
[a_{\overline{\lambda}}b]\ast_{\overline{\lambda}+\overline{\mu}}c)\\
&&+T_r([a_{\overline{\lambda}}b]\circ_{\overline{\lambda}+\overline{\mu}}c)+ \mu_r([a_{\overline{\lambda}}b]\ast_{\overline{\lambda}+\overline{\mu}}c-[(a\circ_{\overline{\lambda}}b)\ast_{\overline{\lambda}+\overline{\mu}}c])
+[[a_{\overline{\lambda}}b]_{\overline{\lambda}+\overline{\mu}}c].
\end{eqnarray*}

By (\ref{a8}), comparing the coefficients of $T_r^2$, $\lambda_rT_r$, $\mu_rT_r$, $\lambda_r\mu_r$,
$\lambda_r^2$, $\mu_r^2$, $T_r$, $\lambda_r$, $\mu_r$ and $1$, we obtain the following equalities:
\begin{eqnarray}\label{c1}
a\circ_{\overline{\lambda}}(b\circ_{\overline{\mu}}c)-(-1)^{\alpha\beta}b\circ_{\overline{\mu}}(a\circ_{\overline{\lambda}}c)=0,\end{eqnarray}
\begin{eqnarray}\label{c2}a\circ_{\overline{\lambda}}(b\circ_{\overline{\mu}}c)+a\ast_{\overline{\lambda}}(b\circ_{\overline{\mu}}c)+(a\circ_{\overline{\lambda}}b)\circ_{\overline{\lambda}+\overline{\mu}}c
-(a\ast_{\overline{\lambda}}b)\circ_{\overline{\lambda}+\overline{\mu}}c\end{eqnarray}
$$-(-1)^{\alpha\beta}b\circ_{\overline{\mu}}(a\ast_{\overline{\lambda}}c)=0,$$
\begin{eqnarray}\label{c3}a\circ_{\overline{\lambda}}(b\ast_{\overline{\mu}}c)+(a\circ_{\overline{\lambda}}b)\circ_{\overline{\lambda}+\overline{\mu}}c-(-1)^{\alpha\beta}(
b\circ_{\overline{\mu}}(a\circ_{\overline{\lambda}}c)\end{eqnarray}
$$+b\ast_{\overline{\mu}}(a\circ_{\overline{\lambda}}c))=0,$$
\begin{eqnarray}\label{c4}a\ast_{\overline{\lambda}}(b\ast_{\overline{\mu}}c)+2(a\circ_{\overline{\lambda}}b)\ast_{\overline{\lambda}+\overline{\mu}}c-(a\ast_{\overline{\lambda}}b)\ast_{\overline{\lambda}+\overline{\mu}}c\end{eqnarray}
$$-(-1)^{\alpha\beta}b\ast_{\overline{\mu}}(a\ast_{\overline{\lambda}}c)=0,$$
\begin{eqnarray}\label{c5}a\ast_{\overline{\lambda}}(b\circ_{\overline{\mu}}c)-(a\ast_{\overline{\lambda}}b)\ast_{\overline{\lambda}+\overline{\mu}}c+(a\circ_{\overline{\lambda}}b)\ast_{\overline{\lambda}+\overline{\mu}}c=0,\end{eqnarray}
\begin{eqnarray}\label{c6}(a\circ_{\overline{\lambda}}b)\ast_{\overline{\lambda}+\overline{\mu}}c-(-1)^{\alpha\beta}b\ast_{\overline{\mu}}(a\circ_{\overline{\lambda}}c)=0,\end{eqnarray}
\begin{eqnarray}\label{c7}
[a_{\overline{\lambda}}(b\circ_{\overline{\mu}}c)]+a\circ_{\overline{\lambda}}[b_{\overline{\mu}}c]-[a_{\overline{\lambda}}b]\circ_{\overline{\lambda}+\overline{\mu}}c\end{eqnarray}
$$-(-1)^{\alpha\beta}
([b_{\overline{\mu}}(a\circ_{\overline{\lambda}}c)]+b\circ_{\overline{\mu}}[a_{\overline{\lambda}}c])=0,$$
\begin{eqnarray}\label{c8}[a_{\overline{\lambda}}(b\circ_{\overline{\mu}}c)]+a\ast_{\overline{\lambda}}[b_{\overline{\mu}}c]+
[(a\circ_{\overline{\lambda}}b)_{\overline{\lambda}+\overline{\mu}}c]-[(a\ast_{\overline{\lambda}}b)_{\overline{\lambda}+\overline{\mu}}c]
\end{eqnarray}
$$-[a_{\overline{\lambda}}b]\ast_{\overline{\lambda}+\overline{\mu}}c-(-1)^{\alpha\beta}[b_{\overline{\mu}}(a\ast_{\overline{\lambda}}c)]=0,$$
\begin{eqnarray}\label{c9}[a_{\overline{\lambda}}(b\ast_{\overline{\mu}}c)]+[(a\circ_{\overline{\lambda}}b)_{\overline{\lambda}+\overline{\mu}}c]
-[a_{\overline{\lambda}}b]\ast_{\overline{\lambda}+\overline{\mu}}c\end{eqnarray}
$$-(-1)^{\alpha\beta}(b\ast_{\overline{\mu}}[a_{\overline{\lambda}}c]+[b_{\overline{\mu}}(a\circ_{\overline{\lambda}}c)])=0,$$
\begin{eqnarray}\label{c10}[a_{\overline{\lambda}}[b_{\overline{\mu}}c]]-[[a_{\overline{\lambda}}b]_{\overline{\lambda}+\overline{\mu}}c]-(-1)^{\alpha\beta}[b_{\overline{\mu}}[a_{\overline{\lambda}}c]]=0.\end{eqnarray}

By (\ref{bb4}) and (\ref{c10}), we know that $(\overline{R},[\cdot_{\overline{\lambda}}\cdot]$ is an $(r-1)$-dim Lie conformal superalgebra.

We first consider (\ref{c3}). Taking (\ref{bbb3}) into (\ref{c3}) and according to (\ref{c1}), we can get
\begin{eqnarray}\label{d1}
a\circ_{\overline{\lambda}}(b_{\overline{\mu}}c)-(a\circ_{\overline{\lambda}}b)_{\overline{\lambda}+\overline{\mu}}c
=(-1)^{\beta\gamma}(a\circ_{\overline{\lambda}}(c_{-\overline{\mu}-\overline{\mathbf{T}}}b)-(a_{\overline{\lambda}}c)_{-\overline{\mu}-\overline{\mathbf{T}}}b).
\end{eqnarray}
By (\ref{c1}) and (\ref{d1}), we know that $(\overline{R},\cdot\circ_{\overline{\lambda}}\cdot)$ is an $(r-1)$-dim right Novikov conformal superalgebra.

Next, we show that when (\ref{c1}) and (\ref{d1}) are satisfied, (\ref{c2}), (\ref{c4}), (\ref{c5}) and (\ref{c6}) holds. Here, we only prove that (\ref{c4}) holds. Other equalities can be checked similarly. The right side of (\ref{c4}) is equal to the following:
\begin{eqnarray*}
&&a\circ_{\overline{\lambda}}(b\ast_{\overline{\mu}}c)
+(-1)^{\alpha(\beta+\gamma)}(b\ast_{\overline{\mu}}c)\circ_{-\overline{\lambda}-\overline{\mathbf{T}}}a
+2(a\circ_{\overline{\lambda}}b)\circ_{\overline{\lambda}+\overline{\mu}}c\\
&&+2(-1)^{(\alpha+\beta)\gamma}c\circ_{-\overline{\lambda}-\overline{\mu}
-\overline{\mathbf{T}}}(a\circ_{\overline{\lambda}}b)
-(a\ast_{\overline{\lambda}}b)\circ_{\overline{\lambda}+\overline{\mu}}c\\
&&-(-1)^{\gamma(\alpha+\beta)}c\circ_{-\overline{\lambda}-\overline{\mu}
-\overline{\mathbf{T}}}(a\ast_{\overline{\lambda}}b)-(-1)^{\alpha\beta}b\circ_{\overline{\mu}}(a\ast_{\overline{\lambda}}c)\\
&&-(-1)^{\beta\gamma}(a\ast_{\overline{\lambda}}c)\circ_{-\overline{\mu}-\overline{\mathbf{T}}}b\\
&&=a\circ_{\overline{\lambda}}(b\circ_{\overline{\mu}}c)
+(-1)^{\beta\gamma}a\circ_{\overline{\lambda}}(c\circ_{-\overline{\mu}-\overline{\mathbf{T}}}b)
+(-1)^{\alpha(\beta+\gamma)}(b\circ_{\overline{\mu}}c)\circ_{-\overline{\lambda}-\overline{\mathbf{T}}}a\\
&&+(-1)^{\alpha(\beta+\gamma)+\beta\gamma}(c\circ_{-\overline{\mu}-\overline{\mathbf{T}}}b)\circ_{-\overline{\lambda}-\overline{\mathbf{T}}}a
+2(a\circ_{\overline{\lambda}}b)\circ_{\overline{\lambda}+\overline{\mu}}c\\
&&+2(-1)^{(\alpha+\beta)\gamma}c\circ_{-\overline{\lambda}-\overline{\mu}
-\overline{\mathbf{T}}}(a\circ_{\overline{\lambda}}b)
-(a\circ_{\overline{\lambda}}b)\circ_{\overline{\lambda}+\overline{\mu}}c
-(-1)^{\alpha\beta}(b\circ_{-\overline{\lambda}-\overline{\mathbf{T}}}a)\circ_{\overline{\lambda}+\overline{\mu}}c\\
&&-(-1)^{\gamma(\alpha+\beta)}c\circ_{-\overline{\lambda}-\overline{\mu}
-\overline{\mathbf{T}}}(a\circ_{\overline{\lambda}}b)
-(-1)^{\gamma(\alpha+\beta)+\alpha\beta}c\circ_{-\overline{\lambda}-\overline{\mu}-\overline{\mathbf{T}}}(b\circ_{-\overline{\lambda}-\overline{\mathbf{T}}}a)
\\
&&-(-1)^{\alpha\beta}b\circ_{\overline{\mu}}(a\circ_{\overline{\lambda}}c)-(-1)^{\alpha(\beta+\gamma)}b\circ_{\overline{\mu}}(c\circ_{-\overline{\lambda}-\overline{T}}a)\\
&&-(-1)^{\beta\gamma}(a\circ_{\overline{\lambda}}c)\circ_{-\overline{\mu}-\overline{\mathbf{T}}}b
-(-1)^{\gamma(\alpha+\beta)}(c\circ_{-\overline{\lambda}-\overline{\mathbf{T}}}a)\circ_{-\overline{\mu}-\overline{\mathbf{T}}}b\\
&&=(-1)^{(\alpha+\beta)\gamma}(c\circ_{-\overline{\lambda}-\overline{\mu}-\overline{\mathbf{T}}}(a\circ_{\overline{\lambda}}b)
-(c\circ_{-\overline{\lambda}-\overline{\mathbf{T}}}a)\circ_{-\overline{\mu}-\overline{\mathbf{T}}}b\\
&&-(-1)^{\alpha\beta}
c\circ_{-\overline{\lambda}-\overline{\mu}-\overline{\mathbf{T}}}(b\circ_{-\overline{\lambda}-\overline{\mathbf{T}}}a)
+(-1)^{\alpha\beta}(c\circ_{-\overline{\mu}-\overline{\mathbf{T}}}b)\circ_{-\overline{\lambda}-\overline{\mathbf{T}}}a)\\
&&-(a\circ_{\overline{\lambda}}(b\circ_{\overline{\mu}}c)
-(a\circ_{\overline{\lambda}}b)\circ_{\overline{\lambda}+\overline{\mu}}c
-(-1)^{\beta\gamma}a\circ_{\overline{\lambda}}(c\circ_{-\overline{\mu}-\overline{\mathbf{T}}}b)\\
&&+(-1)^{\beta\gamma}(a\circ_{\overline{\lambda}}c)\circ_{-\overline{\mu}-\overline{\mathbf{T}}}b)
+(-1)^{\alpha\beta}(b\circ_{\overline{\mu}}(a\circ_{\overline{\lambda}}c)
-(b\circ_{-\overline{\lambda}-\overline{\mathbf{T}}}a)\circ_{\overline{\lambda}+\overline{\mu}}c\\
&&-(-1)^{\alpha\gamma}b\circ_{\overline{\mu}}(c\circ_{-\overline{\lambda}-\overline{\mathbf{T}}}a)+
(-1)^{\alpha\gamma}(b\circ_{\overline{\mu}}c)\circ_{-\overline{\lambda}-\overline{\mathbf{T}}}a)\\
&&=0.
\end{eqnarray*}

Finally, we show that (\ref{c8}) and (\ref{c9}) can be obtained from (\ref{c7}). Here, we only prove that (\ref{c8}) holds. (\ref{c9}) can be checked similarly. Using (\ref{bb4}) and (\ref{c7}), the right side of (\ref{c8}) is as follows:
\begin{eqnarray*}
&&[a_{\overline{\lambda}}(b\circ_{\overline{\mu}}c)]+a\circ_{\overline{\lambda}}[b_{\overline{\mu}}c]
+(-1)^{\alpha(\beta+\gamma)}[b_{\overline{\mu}}c]\circ_{-\overline{\lambda}-\overline{\mathbf{T}}}a+
[(a\circ_{\overline{\lambda}}b)_{\overline{\lambda}+\overline{\mu}}c]\\
&&-[(a\circ_{\overline{\lambda}}b)_{\overline{\lambda}+\overline{\mu}}c]
-(-1)^{\alpha\beta}[(b\circ_{-\overline{\lambda}-\overline{\mathbf{T}}}a)_{\overline{\lambda}+\overline{\mu}}c]
-[a_{\overline{\lambda}}b]\circ_{\overline{\lambda}+\overline{\mu}}c\\
&&-(-1)^{(\alpha+\beta)\gamma}c\circ_{-\overline{\lambda}-\overline{\mu}-\overline{\mathbf{T}}}[a_{\overline{\lambda}}b]
-(-1)^{\alpha\beta}[b_{\overline{\mu}}(a\circ_{\overline{\lambda}}c)]
-(-1)^{\alpha(\beta+\gamma)}[b_{\overline{\mu}}(c\circ_{-\overline{\lambda}-\overline{\mathbf{T}}}a)]\\
&&=([a_{\overline{\lambda}}(b\circ_{\overline{\mu}}c)]+a\circ_{\overline{\lambda}}[b_{\overline{\mu}}c]
-[a_{\overline{\lambda}}b]\circ_{\overline{\lambda}+\overline{\mu}}c
-(-1)^{\alpha\beta}[b_{\overline{\mu}}(a\circ_{\overline{\lambda}}c)])\\
&&+(-1)^{\alpha(\beta+\gamma)}[b_{\overline{\mu}}c]\circ_{-\overline{\lambda}-\overline{\mathbf{T}}}a
+(-1)^{\alpha\beta+(\alpha+\beta)\gamma)}[c_{-\overline{\lambda}-\overline{\mu}-\overline{\mathbf{T}}}(b\circ_{\overline{\lambda}-\overline{\mathbf{T}}}a)]\\
&&+(-1)^{(\alpha+\beta)\gamma+\alpha\beta}c\circ_{-\overline{\lambda}-\overline{\mu}-\overline{\mathbf{T}}}[b_{-\overline{\lambda}-\overline{\mathbf{T}}}a]
-(-1)^{\alpha(\beta+\gamma)}[b_{\overline{\mu}}(c\circ_{-\overline{\lambda}-\overline{\mathbf{T}}}a)]\\
&&=-(-1)^{\alpha(\beta+\gamma)}([b_{\overline{\mu}}(c\circ_{-\overline{\lambda}-\overline{\mathbf{T}}}a)]+
b\circ_{\overline{\mu}}[c_{-\overline{\lambda}-\overline{\mathbf{T}}}a]-
[b_{\overline{\mu}}c]\circ_{-\overline{\lambda}-\overline{\mathbf{T}}}a\\
&&-(-1)^{\beta\gamma}[c_{-\overline{\lambda}-\overline{\mu}-\overline{\mathbf{T}}}(b\circ_{\overline{\lambda}-\overline{\mathbf{T}}}a)]
-(-1)^{\beta\gamma}c\circ_{-\overline{\lambda}-\overline{\mu}-\overline{\mathbf{T}}}[b_{-\overline{\lambda}-\overline{\mathbf{T}}}a])\\
&&=0.
\end{eqnarray*}

Therefore, from the above discussion, the Lie conformal superalgebra structure of $R$ is determined by $(\overline{R},[\cdot_{\overline{\lambda}}\cdot], \cdot\circ_{\overline{\lambda}}\cdot)$ where $(\overline{R},[\cdot_{\overline{\lambda}}\cdot])$ is an $(r-1)$-dim Lie conformal superalgebra,
$(\overline{R},\cdot\circ_{\overline{\lambda}}\cdot)$ is an $(r-1)$-dim right Novikov conformal superalgebra, and they satisfy (\ref{c7}).

For the unity with Theorem \ref{TTa1}, set
\begin{eqnarray}\label{aa1} a_{\overline{\lambda}}b=(-1)^{\alpha\beta}(b\circ_{-\overline{\lambda}-\overline{\mathbf{T}}}a)
\end{eqnarray} for
$a\in \overline{R}_{\alpha}$, $b\in \overline{R}_{\beta}$. Therefore, $(\overline{R},\cdot_{\overline{\lambda}}\cdot)$ is an $(r-1)$-dim Novikov conformal algebra. By (\ref{c7}), using (\ref{aa1}), we can obtain
the following equality:
\begin{eqnarray}\label{aa2}
[(c_{-\overline{\mathbf{\mu}}-\overline{\mathbf{T}}}b)_{-\overline{\mathbf{\lambda}}-\overline{\mathbf{T}}}a]
+[c_{-\overline{\mathbf{\mu}}-\overline{\mathbf{T}}}b]_{-\overline{\mathbf{\lambda}}-\overline{\mathbf{T}}}a
-c_{-\overline{\mathbf{\lambda}}-\overline{\mathbf{\mu}}-\overline{\mathbf{T}}}[b_{-\overline{\mathbf{\lambda}}-\overline{\mathbf{T}}}a]\end{eqnarray}
$$-(-1)^{\alpha\beta}[(c_{-\overline{\mathbf{\lambda}}-\overline{\mathbf{T}}}a)_{-\overline{\mathbf{\mu}}-\overline{\mathbf{T}}}b]
-(-1)^{\alpha\beta}[c_{-\overline{\mathbf{\lambda}}-\overline{\mathbf{T}}}a]_{-\overline{\mathbf{\mu}}-\overline{\mathbf{T}}}b=0,$$
for $a\in \overline{R}_{\alpha}$, $b\in \overline{R}_{\beta}$, $c\in \overline{R}$.
Therefore, by (\ref{aa2}), we obtain that $(\overline{R},[\cdot_{\overline{\lambda}}\cdot], \cdot_{\overline{\lambda}}\cdot)$ is an $(r-1)$-dim
super Gel'fand-Dorfman conformal bialgebra.

Thus, the theorem holds when $i=r$. So, the proof is finished.
\end{proof}

By Theorem \ref{TT1}, we know that one $r$-dim super Gel'fand-Dorfman conformal bialgebra can correspond to $r$ $(r+1)$-dim Lie conformal superalgebras.  In the following, we only consider the $r$-dim $r$-linear Lie conformal superalgebras.

Next, we give two remarks about $r$-dim super Gel'fand-Dorfman conformal bialgebras.

\begin{remark}\label{ra1}
In fact, there are some relations between $r$-dim super Gel'fand-Dorfman conformal bialgebra and super Gel'fand-Dorfman bialgebra. If $(R=\mathbb{C}[\mathbf{T}]V,[\cdot_\lambda \cdot],
\cdot_\lambda \cdot)$ is an $r$-dim super Gel'fand-Dorfman conformal bialgebra, then $(\text{Lie} R,[\cdot,\cdot],\circ)$ is a super Gel'fand-Dorfman bialgebra with the Lie bracket and the Novikov algebraic operation as follows:
\begin{eqnarray}
[a_{\mathbf{m}}, b_\mathbf{n}]=\sum_{\mathbf{j}\in \mathbb{Z}_{+}^r}\left(\begin{array}{ccc}
\mathbf{m}\\\mathbf{j}\end{array}\right)(a_{(\mathbf{j})}b)_{\mathbf{m}+\mathbf{n}-\mathbf{j}},\\
a_{\mathbf{m}}\circ b_\mathbf{n}=\sum_{\mathbf{j}\in \mathbb{Z}_{+}^r}\left(\begin{array}{ccc}
\mathbf{m}\\\mathbf{j}\end{array}\right)(a_{\mathbf{j}}b)_{\mathbf{m}+\mathbf{n}-\mathbf{j}},
\end{eqnarray}
where $[a_{\mathbf{\lambda}}b]=\sum_{\mathbf{m}\in \mathbb{Z}_{+}^r}\mathbf{\lambda}^{(\mathbf{m})}a_{(\mathbf{m})}b$ and
$a_{\mathbf{\lambda}}b=\sum_{\mathbf{m}\in \mathbb{Z}_{+}^r}\mathbf{\lambda}^{(\mathbf{m})}a_{\mathbf{m}}b$. The proof is similar to that in Section 2.7 in \cite{K1}.
\end{remark}

\begin{remark}
Obviously, any $r$-dim Novikov conformal superalgebra $(R=\mathbb{C}[\mathbf{T}]V,$ $\cdot_\lambda \cdot)$ with the trivial $r$-dim Lie conformal superalgebra can construct a super Gel'fand-Dorfman conformal bialgebra. Therefore, by Theorem \ref{TT1}, given an $r$-dim Novikov conformal superalgebra $(R=\mathbb{C}[\mathbf{T}]V,$ $\cdot_\lambda \cdot)$, we can obtain an $(r+1)$-dim Lie conformal superalgebra $(\widetilde{R}=\mathbb{C}[T_1,\cdots,T_r, T_{r+1}]V, [\cdot_{\widetilde{\lambda}} \cdot])$  with the $\lambda$-bracket as follows:
\begin{eqnarray}
[a_{\widetilde{\lambda}}b]=T_{r+1}((-1)^{\alpha\beta}b_{-\lambda-\mathbf{T}}a)+\lambda_{r+1}((-1)^{\alpha\beta}b_{-\lambda-\mathbf{T}}a
+a_{\lambda}b),
\end{eqnarray}
where $\widetilde{\lambda}=(\lambda_1,\cdots,\lambda_r,\lambda_{r+1})$, $a\in V_{\alpha}$, $b\in V_{\beta}$.
\end{remark}
\begin{proposition}\label{pp1}
Let $(R,{\cdot}_{\mathbf{\lambda}}{\cdot})$ be an $r$-dim  Novikov conformal superalgebra. Then,
$(R,{\cdot}_{\mathbf{\lambda}}{\cdot},[{\cdot}_{\mathbf{\lambda}}{\cdot}])$ is an $r$-dim super Gel'fand-Dorfman conformal bialgebra with the
$\mathbf{\lambda}$-bracket given as follows:
\begin{eqnarray}\label{ccc2}
[a_{\mathbf{\lambda}}b]=a_{\mathbf{\lambda}}b-(-1)^{\alpha\beta}b_{-\mathbf{\lambda}-\mathbf{T}}a, ~~~a\in R_{\alpha},~~b\in R_{\beta}.
\end{eqnarray}
\end{proposition}
\begin{proof}
By Proposition \ref{ppp1}, we only need to check that (\ref{ccc1}) hold. By (\ref{ccc2}),
the right side of (\ref{ccc1}) is the following:
\begin{eqnarray*}
&&(a_{-\mu-\mathbf{T}}b)_{-\lambda-\mathbf{T}}c-(-1)^{(\alpha+\beta)\gamma}c_{\lambda}(a_{-\mu-\mathbf{T}}b)
+(a_{-\mu-\mathbf{T}}b)_{-\lambda-\mathbf{T}}c\\
&&-(-1)^{\alpha\beta}(b_{\mu}a)_{-\lambda-\mathbf{T}}c
-a_{-\lambda-\mu-\mathbf{T}}(b_{-\lambda-\mathbf{T}}c)
+(-1)^{\beta\gamma}a_{-\lambda-\mu-\mathbf{T}}(c_{\lambda}b)\\
&&-(-1)^{\beta\gamma}(a_{-\lambda-\mathbf{T}}c)_{-\mu-\mathbf{T}}b
+(-1)^{\alpha\beta}b_{\mu}(a_{-\lambda-\mathbf{T}}c)\\
&&-(-1)^{\beta\gamma}(a_{-\lambda-\mathbf{T}}c)_{-\mu-\mathbf{T}}b+(-1)^{(\alpha+\beta)\gamma}(c_{\lambda}a)_{-\mu-\mathbf{T}}b\\
&=&((a_{-\mu-\mathbf{T}}b)_{-\lambda-\mathbf{T}}c-a_{-\lambda-\mu-\mathbf{T}}(b_{-\lambda-\mathbf{T}}c)-(-1)^{\alpha\beta}(b_{\mu}a)_{-\lambda-\mathbf{T}}c
\\
&&+(-1)^{\alpha\beta}b_{\mu}(a_{-\lambda-\mathbf{T}}c))-(-1)^{\beta\gamma}((a_{-\lambda-\mathbf{T}}c)_{-\mu-\mathbf{T}}b-a_{-\lambda-\mu-\mathbf{T}}(c_{\lambda}b)\\
&&+(-1)^{\alpha\gamma}c_{\lambda}(a_{-\mu-\mathbf{T}}b)-(-1)^{\alpha\gamma}(c_{\lambda}a)_{-\mu-\mathbf{T}}b)\\
&&+((a_{-\mu-\mathbf{T}}b)_{-\lambda-\mathbf{T}}c-(-1)^{\beta\gamma}(a_{-\lambda-\mathbf{T}}c)_{-\mu-\mathbf{T}}b)\\
&=&0.
\end{eqnarray*}
Therefore, $(R,{\cdot}_{\mathbf{\lambda}}{\cdot},[{\cdot}_{\mathbf{\lambda}}{\cdot}])$ is an $r$-dim super Gel'fand-Dorfman conformal bialgebra.
\end{proof}

By Propositions \ref{pp1} and \ref{pp2}, we can easily obtain the following corollary.
\begin{corollary}\label{cccc1}
If $(V, \circ, \cdot)$ is a Novikov-Poisson superalgebra, then $R=\mathbb{C}[T]V$ is a super Gel'fand-Dorfman conformal algebra with the following $\lambda$-product and $\lambda$-bracket $(a\in V_\alpha, b\in V_\beta)$:
\begin{eqnarray}
&&a_\lambda b=(\lambda+T)(a\cdot b)+a\circ b,\\
&& [a_\lambda b]=(T+2\lambda)(a\cdot b)+(a\circ b-(-1)^{\alpha\beta}b\circ a).
\end{eqnarray}
\end{corollary}

\begin{remark}
In fact, by Corollary \ref{cccc1} and Remark \ref{ra1}, we can obtain a super Gel'fand-Dorfman bialgebra from a Novikov-Poisson superalgebra.
The construction is as follows.

Let $(V, \circ, \cdot)$ is a Novikov-Poisson superalgebra, then $(A=V\otimes \mathbf{C}[t,t^{-1}],\ast,[\cdot,\cdot])$
is a super Gel'fand-Dorfman bialgebra with the following Novikov superalgebra structure and the Lie superalgebra structure:
\begin{eqnarray*}
&&(a\otimes t^m) \ast (b\otimes t^n)=(a\circ b)\otimes t^{m+n}-n(a\cdot b)\otimes t^{m+n-1},\\
&&[a\otimes t^m,b\otimes t^n]=(m-n)(a\cdot b)\otimes t^{m+n-1}-(a\circ b-(-1)^{\alpha\beta}b\circ a)\otimes t^{m+n},
\end{eqnarray*}
where $a\in V_{\alpha}$, $b\in V_{\beta}$, $m$, $n\in \mathbb{Z}$.

\end{remark}
Next, we present several examples of super Gel'fand-Dorfman conformal bialgebras.
\begin{example}
By Proposition \ref{pp3} and Proposition \ref{pp1}, we have the following Gel'fand-Dorfman conformal bialgebra $R=\mathbb{C}[T]x$ with the $\lambda$-product and $\lambda$-bracket as follows ($a\in \mathbb{C}$):
\begin{eqnarray*}
&&x_\lambda x=(\lambda+T+a)x,\\
&&[x_\lambda x]=(T+2\lambda)x.
\end{eqnarray*}

This Gel'fand-Dorfman conformal bialgebra corresponds to a 2-dim $2$-linear Lie conformal algebra
$\widetilde{R}=\mathbb{C}[T_1,T_2]x$ with the following $\lambda$-bracket:
$$[x_\lambda x]=T_2(-\lambda_1+a)x+\lambda_2(T_1+2a)x+(T_1+2\lambda_1)x.$$

$\text{Lie}\widetilde{R}=\mathbb{C}x\otimes \mathbb{C}[t_1^{\pm 1},t_2^{\pm 1}]$ is a Lie algebra with the following Lie bracket:
\begin{eqnarray*}
&&[x\otimes t_1^{m_1}t_2^{n_1}, x\otimes t_1^{m_2}t_2^{n_2}]\\
&&=a(n_1-n_2)x\otimes t_1^{m_1+m_2}t_2^{n_1+n_2-1}
+(m_1-m_2)x\otimes t_1^{m_1+m_2-1}t_2^{n_1+n_2}\\
&&+(m_1n_2-n_1m_2)x\otimes t_1^{m_1+m_2-1}t_2^{n_1+n_2-1},
\end{eqnarray*}
for $m_1$, $m_2$, $n_1$, $n_2\in \mathbb{Z}$.
\end{example}
\begin{example}
Let $R=R_{\overline{0}}\oplus R_{\overline{1}}=\mathbb{C}[T]a\oplus\mathbb{C}[T]b$ be the Novikov conformal superalgebra given in
Example \ref{ex1}. Then, by Proposition \ref{pp1}, we obtain a super Gel'fand-Dorfman bialgebra $(R, \cdot_\lambda \cdot, [\cdot_\lambda \cdot])$ given by as follows:
\begin{eqnarray*}
&&a_{\lambda}a=(\lambda+T+C_1)a,~~~a_{\lambda}b=(\lambda+T+C_1)b,~~~b_{\lambda}a=(\lambda+T+C_2)b,~~b_{\lambda}b=0,\\
&&[a_{\lambda} a]=(T+2\lambda)a,~~~[a_{\lambda} b]=(T+2\lambda+C_1-C_2)b,~~~[b_{\lambda}b]=0,
\end{eqnarray*}
for $C_1$, $C_2\in \mathbb{C}$.

This Gel'fand-Dorfman conformal bialgebra corresponds to a 2-dim $2$-linear Lie conformal algebra $\widetilde{R}=\mathbb{C}[T_1,T_2]a \oplus\mathbb{C}[T_1,T_2]b$ with the following $\lambda$-bracket:
\begin{eqnarray*}
&&[a_{\lambda}a]=T_2(-\lambda+C_1)a+\lambda_2(T_1+2C_1)a+(T_1+2\lambda_1)a,\\
&&[a_{\lambda}b]=T_2(-\lambda+C_2)b+\lambda_2(T_1+C_1+C_2)b+(T_1+2\lambda_1+C_1-C_2)b,\\
&&[b_{\lambda}b]=0,
\end{eqnarray*}
for $C_1$, $C_2\in \mathbb{C}$.
\end{example}
\begin{example}
If $(A,\cdot)$ is a commutative associative algebra and $D: A\rightarrow A$ is a derivation of $A$, then $(A, \circ, \cdot)$ is a Novikov-Poisson algebra with
$$a\circ b=a\cdot Db.$$
Then, by Corollary \ref{cccc1}, we obtain a $2$-dim $2$-linear Lie conformal algebra $R=\mathbb{C}[T_1,T_2]A$ in 2 dimension with the following $\lambda$-product and $\lambda$-bracket:
\begin{eqnarray*}
[a_\lambda b]&=&T_2(-\lambda_1(a\cdot b)+b\cdot Da)+\lambda_2(T_1(a\cdot b)+b\cdot Da+a\cdot Db)\\
&&+(T_1+2\lambda_1)(a\cdot b)+(a\cdot Db-b\cdot Da), ~~~a,~b\in A.
\end{eqnarray*}

\end{example}

Next, we present another construction of Gel'fand-Dorfman conformal bialgebras.
\begin{proposition}
If $(V,[\cdot,\cdot],\circ)$ is a Gel'fand-Dorfman bialgebra, $(V,\circ,\cdot)$ is Novikov-Poisson algebra and $(V, [\cdot,\cdot],\cdot)$ is a Lie-Poisson algebra, then $R=\mathbb{C}[T]V$ is a Gel'fand-Dorfman conformal bialgebra with the following $\lambda$-product and $\lambda$-bracket ($a\in V_\alpha, b\in V_\beta$):
\begin{eqnarray*}
&&a_\lambda b=(\lambda+T)(a\cdot b)+a\circ b,\\
&&[a_\lambda b]=T(b\circ a)+\lambda(a\circ b+b\circ a)+[b,a].
\end{eqnarray*}
The quadruple $(V, [\cdot,\cdot], \circ, \cdot)$ is called \emph{Gel'fand-Dorfman Novikov-Poisson algebra}.
\end{proposition}
\begin{proof}
It can be check directly.
\end{proof}

\begin{remark}
If the operation $``\cdot"$ in the above proposition is trival, then we can obtain a construction of  Gel'fand-Dorfman conformal bialgebras from Gel'fand-Dorfman bialgebras.

If $(V,[\cdot,\cdot],\circ)$ is a Gel'fand-Dorfman bialgebra, then $R=\mathbb{C}[T]V$ is a
Gel'fand-Dorfman conformal bialgebra with the following $\lambda$-product and $\lambda$-bracket ($a\in V_\alpha, b\in V_\beta$):
\begin{eqnarray*}
&&a_\lambda b=a\circ b,\\
&&[a_\lambda b]=T(b\circ a)+\lambda(a\circ b+b\circ a)+[b,a].
\end{eqnarray*}

In fact, by Remark \ref{ra1}, we can also obtain a construction of Gel'fand-Dorfman bialgebras as follows.

If $(V,[\cdot,\cdot],\circ)$ is a Gel'fand-Dorfman bialgebra, then $(\text{Lie}R=V\otimes \mathbb{C}[t,t^{-1}],$ $\ast,$ $ [\cdot,\cdot])$ is a Gel'fand-Dorfman bialgebra with the Novikov operation and Lie algebra operation as follows ($a\in V_\alpha, b\in V_\beta$, $m$, $n\in \mathbb{C}$):
\begin{eqnarray*}
&&a\otimes t^m \ast b\otimes t^n=(a\circ b)\otimes t^{m+n},\\
&&[a\otimes t^m, b\otimes t^n]=(ma\circ b-nb\circ a)\otimes t^{m+n-1}+[b,a]\otimes t^{m+n},
\end{eqnarray*}

\end{remark}

We give a construction of Gel'fand-Dorfman Novikov-Poisson algebras in the following.
\begin{example}
Let $(A,\cdot,[\cdot,\cdot])$ be a Lie-Poisson algebra and let $d$ be a derivation of the algebra $(A,\cdot)$ such that
\begin{eqnarray}
d[u,v]=[d(u),v]+[u,d(v)]+\xi[u,v],~~~~~\text{for}~~~u,~~v\in A,
\end{eqnarray}
where $\xi\in \mathbb{C}$ is a constant. We define another algebraic operation $\circ$ on $A$ by
\begin{eqnarray}
u\circ v=u\cdot d(v)+\xi u\cdot v ~~~~~~~\text{for}~~~u,~~v\in A.
\end{eqnarray}
By Corollary 2.6 in \cite{X3}, $(A,\cdot,\circ)$ is a Novikov-Poisson algebra. And, by Theorem 3.2 in \cite{Xu1}, $(A,[\cdot,\cdot],\circ)$ forms a Gel'fand-Dorfman bialgebra. Therefore,  $(V, [\cdot,\cdot], \circ, \cdot)$ is a Gel'fand-Dorfman Novikov-Poisson algebra.
\end{example}

Finally, we will present a theorem about an equivalent characterization of $r$-dim linear Lie conformal superalgebra.
\begin{theorem}\label{T1}
An $r$-dim linear Lie conformal superalgebra $R=\mathbb{C}[\mathbf{T}]V$ is equivalent to that $V=V_{\overline{0}}\oplus V_{\overline{1}}$ is equipped with one Lie superalgebraic operation
$[\cdot,\cdot]$, and $r$ Novikov superalgebraic operations $\circ_i$ such that $(V, [\cdot,\cdot], \circ_i)$ is a super Gel'fand-Dorman bialgebra for each $1\leq i\leq r$ and $\circ_i$ and $\circ_j$ satisfying the following conditions:
\begin{eqnarray}
&\label{T2}a\circ_i(b\circ_jc)-(a\circ_i b)\circ_jc=(-1)^{\alpha\beta}[b\circ_j(a\circ_ic)-(b\circ_j a)\circ_i c],\\
&\label{T3}(c\circ_i a)\circ_j b+(c\circ_j a)\circ_i b=(-1)^{\alpha\beta}[(c\circ_i b)\circ_ja+(c\circ_j b)\circ_ia],
\end{eqnarray}
for $a\in V_{\alpha}$, $b\in V_{\beta}$ , $c\in V$ and $i\neq j$. We can such $(V, [\cdot,\cdot], \circ_1,\cdots,\circ_r)$ a \emph{generalized super Gel'fand-Dorfman algebra}.
\end{theorem}
\begin{proof}
By (\ref{a4}), we can assume that:
\begin{eqnarray}\label{a5}
[a_{\mathbf{\lambda}}b]=\sum_{i=1}^rT_i(b\circ_ia)+\sum_{j=1}^r\lambda_j(b\ast_ja)+[b,a]~~~\text{for}~~a,~b\in V,
\end{eqnarray}
where $\circ_i$, $\ast_j$, $[\cdot,\cdot]$ are algebraic operations on $V$.
Then, the proof is similar to that in Theorem \ref{TT1}. Here, $b\ast_ia=(b\circ_i a)+(-1)^{\alpha\beta}a\circ_i b$, where $a\in V_{\alpha}$,
$b\in V_{\beta}$.
\end{proof}

\begin{remark}
If $(V,[\cdot,\cdot],\circ_i)$ is a super Gel'fand-Dorfman bialgebra for one of $1\leq i\leq r$, then $(V,[\cdot,\cdot],\circ_1, \cdots, \circ_r)$ with $\circ_j$ trivial for each $1\leq j\leq r$ and $j\neq i$ is a generalized super Gel'fand-Dorfman algebra, and its corresponding $r$-dim Lie conformal superalgebra $R=\mathbb{C}[\mathbf{T}]V$ is determined the following $\mathbf{\lambda}$-bracket:
$$[a_{\mathbf{\lambda}}b]=T_i(b\circ_ia)+\lambda_i(b\circ_ia+(-1)^{\alpha\beta}a\circ_ib)+[b,a],$$
for $a\in V_{\alpha}$ and $b\in V_{\beta}$.

In addition, by Theorem \ref{T1}, we can classify $r$-dim linear Lie conformal superalgebras by classifying generalized super Gel'fand-Dorfman algebras. One useful method to classify generalized super Gel'fand-Dorfman algebras is that given a Lie superalgebra, we first classify all Novikov superalgebra structures compatible with this Lie superalgebra, i.e. constructing a super Gel'fand-Dorfman bialgebra and then study which Novikov superalgebra structures satisfy (\ref{T2}) and (\ref{T3}).
\end{remark}

\begin{example}
Let $V=\mathbb{C}e_1\oplus \mathbb{C}e_2$. Then, $(V,[\cdot,\cdot])$ is a Lie algebra with the Lie bracket $[e_1,e_2]=e_2$.

Set $e_i\circ e_j=\sum_{l=1}^2c_{ij}^le_l$ for any $i$, $j\in\{1,2\}$.  If $[\cdot,\cdot]$ and $\circ$ satisfy
(\ref{xx1}), by some computations, we can obtain
$$c_{11}^1=c_{21}^2, c_{12}^1+c_{21}^1=0, c_{22}^1=0, c_{21}^1=c_{22}^2.$$
For simplicity, set $a=c_{11}^1=c_{21}^2$, $b=c_{11}^2$, $c=c_{12}^2$, $d=c_{21}^1=c_{22}^2=-c_{12}^1$.
Then, if $(V,\circ)$ is a Novikov algebra, we get $cd=bd=0$.

Therefore, $(V,[\cdot,\cdot],\circ)$ is a Gel'fand-Dorfman bialgebra if and only if the Novikov algebraic operation is defined by
$$e_1\circ e_1=ae_1+be_2,~~e_1\circ e_2=-de_1+ce_2,~~e_2\circ e_1=de_1+ae_2,~~e_2\circ e_2=de_2,$$
with $cd=bd=0$.

Assume that $(V,[\cdot,\cdot],\circ_1)$ and $(V,[\cdot,\cdot],\circ_2)$ are two Gel'fand-Dorfman bialgebras with the Novikov algebraic operations
\begin{eqnarray}\label{wq1}
&&e_1\circ_i e_1=a_ie_1+b_ie_2,~e_1\circ_i e_2=-d_ie_1+c_ie_2,\\
&&\label{wq2}e_2\circ_i e_1=d_ie_1+a_ie_2,~e_2\circ_i e_2=d_ie_2,
\end{eqnarray}
with \begin{eqnarray}\label{wq3}
c_id_i=b_id_i=0,~~~~i\in\{1,2\}.
\end{eqnarray} For $\circ_1$ and $\circ_2$ satisfying (\ref{T2}) and (\ref{T3}), by some computations,
we have
\begin{eqnarray}\label{wq4}
a_1c_2=a_2c_1,~~~b_2(c_1-a_1)=b_1(c_2-a_2).
\end{eqnarray}

Therefore, $(V,[\cdot,\cdot],\circ_1,\circ_2)$ is a generalized Gel'fand-Dorfman algebra with the Novikov algebraic operations defined by
(\ref{wq1}),(\ref{wq2}) if and only if (\ref{wq3}) and (\ref{wq4}) hold.

For example, $(V=\mathbb{C}e_1\oplus \mathbb{C}e_2,[\cdot,\cdot],\circ_1,\circ_2)$ is a generalized Gel'fand-Dorfman algebra with the Lie algebra structure and Novikov algebra structures as follows
\begin{eqnarray*}
&&[e_1,e_2]=e_2, \\
&&e_1\circ_1 e_1=e_2,~~e_1\circ_1 e_2=e_2\circ_1 e_1=e_2\circ_1 e_2=0,\\
&&e_1\circ_2 e_1=e_1,~~e_1\circ_2 e_2=e_2,~~e_2\circ_2 e_1=e_2\circ_2 e_2=0.
\end{eqnarray*}
It corresponds to a 2-dim linear Lie conformal algebra $R=\mathbb{C}[T_1,T_2]V$ with the following $\lambda$-bracket
\begin{eqnarray*}
&&[{e_1}_\lambda e_1]=(T_1+2\lambda_1)e_2+(T_2+2\lambda_2)e_1,\\
&&[{e_1}_\lambda e_2]=\lambda_2e_2-e_2,\\
&&[{e_2}_\lambda e_2]=0.
\end{eqnarray*}
\end{example}

\begin{proposition}
Let $(V,[\cdot,\cdot], \circ_i)$ be a super Gel'fand-Dorfman bialgebra for some $i$. Assume that
other operations $\circ_j$ for $1\leq j\leq r$ and $j\neq i$ are defined as follows:
\begin{eqnarray}
a\circ_jb=k_ja\circ_ib,~~\text{for some $k_j\in \mathbb{C}$.}
\end{eqnarray}
Then, $(V,[\cdot,\cdot],\circ_1, \cdots, \circ_r)$ is a generalized super Gel'fand-Dorfman algebra.
\end{proposition}
\begin{proof}
It is easy to check.
\end{proof}

\end{document}